\documentclass[a4paper, oneside, 11pt]{article}

\usepackage[utf8]{inputenc}
\usepackage[T1]{fontenc}
\usepackage[english]{babel}
\usepackage[margin=3cm]{geometry}
\usepackage{amsmath,amssymb,amsfonts,amsthm,latexsym,bbm}
\usepackage{enumerate,enumitem}
\usepackage{graphics,graphicx}
\usepackage{hyperref}
\usepackage[capitalize]{cleveref}
\usepackage{authblk} 

\numberwithin{equation}{section}

\theoremstyle{plain}
\newtheorem{theorem}{Theorem}[section]
\newtheorem{proposition}[theorem]{Proposition}
\newtheorem{lemma}[theorem]{Lemma}
\newtheorem{corollary}[theorem]{Corollary}

\theoremstyle{definition}

\theoremstyle{remark}

\newtheorem{example}[theorem]{Example}
\newtheorem{examples}[theorem]{Examples}

\newcommand{\bbR}{\mathbb{R}}
\newcommand{\bbC}{\mathbb{C}}
\newcommand{\calC}{\mathcal{C}}
\newcommand{\calL}{\mathcal{L}}
\newcommand{\rmC}{\mathrm{C}}

\newcommand{\Implies}[2]{``\ref{#1}~$\Rightarrow$~\ref{#2}''}

\let\epsilon=\varepsilon
\let\phi=\varphi
\let\theta=\vartheta

\newcommand{\one}{\mathbbm{1}}
\newcommand{\norm}[1]{\left\lVert #1 \right\rVert}
\newcommand{\argument}{\,\cdot\,}
\newcommand{\extreme}{\operatorname{ext}} 

\title{A functional representation approach to vector lattice covers for spaces of compact operators}

\author[$\dagger$]{Onno van Gaans}
\author[$\star$]{Jochen Glück}
\author[$\ddag$]{Anke Kalauch}

\affil[$\dagger$]{Leiden University, Einsteinweg 55, 2333CC Leiden, Email: vangaans@math.leidenuniv.nl}
\affil[$\star$]{University of Wuppertal, Gaußstr.\ 20, 42119 Wuppertal, Germany. Email: glueck@uni-wuppertal.de} 
\affil[$\ddag$]{TU Dresden, 01062 Dresden, Email: Anke.Kalauch@tu-dresden.de}

\begin{document}

\maketitle

\begin{abstract}
	For ordered normed vector spaces $X, Y$, we consider the space $\calL(X,Y)$ of bounded linear operators 
	and characterize when its cone of positive operators has non-empty interior. 
	When this is satisfied, we give a functional representation 
	of the closure $\calC(X,Y)$ of the finite rank operators in $\calL(X,Y)$. 
	This space is particularly interesting since it coincides in many cases with the space of compact operators from $X$ to $Y$. 
	
	Our functional representation has very good order properties 
	in the sense that it is a so-called vector lattice cover of $\calC(X,Y)$. 
	This can be used to characterize disjointness of operators in $\calC(X,Y)$ 
	and to determine which operators have a modulus in $\calC(X,Y)$. 
	We demonstrate how our results can be applied to a variety of concrete spaces.
\end{abstract}

\medskip


\section{Introduction}

Let $X$ and $Y$ be two ordered normed spaces  
(for unexplained terminology and notation see Section~\ref{sec:prelim}). 
From them, one obtains new ordered normed spaces by considering various classes of bounded linear operators from $X$ to $Y$. 
It is then natural to ask how properties of those operator spaces can be described in terms of $X$ and $Y$.
We prove the following results about this question.

\begin{enumerate}[label=(\roman*)]
	\item 
	We consider the space $\calL(X,Y)$ of bounded linear operator from $X$ to $Y$ 
	and characterize when its subset $\calL(X,Y)_+$ of positive operators has non-empty interior
	(Theorem~\ref{thm:interior-op-cone}).
	
	\item 
	We consider the operator norm closure of the finite rank operators within $\calL(X,Y)$, which we denote by $\calC(X,Y)$ 
	and which coincides with the space of compact operators whenever $Y$ has the approximation property. 
	Under appropriate assumptions on $X$ and $Y$, 
	we characterize all extremal vectors of the dual cone $\calC(X,Y)'_+$ of $\calC(X,Y)_+$. 
	(Theorem~\ref{thm:extremal-ops}).
	
	\item 
	If $\calL(X,Y)_+$ has non-empty interior and $X$ and $Y$ satisfy appropriate assumptions, 
	we use the aforementioned results to explicitly construct a so-called \emph{vector lattice cover} of $\calC(X,Y)$, 
	i.e., a representation of $\calC(X,Y)$ as a subspace of a vector lattice with certain order density properties 
	(Theorem~\ref{thm:vcl-ops}). 
	
	\item 
	We use the vector lattice cover from Theorem~\ref{thm:vcl-ops} 
	to study disjointness of operators in $\mathcal{C}(X,Y)$ 
	(Corollaries~\ref{cor:disjoint-operators}, \ref{cor:anti-lattice-operators}, and~\ref{cor:no-disjoint-operators}), 
	and we discuss a method to characterize those operators that have a modulus in $\mathcal{C}(X,Y)$.  
\end{enumerate}

The result in~(iii) generalizes the construction of a vector lattice cover 
for the case of $X$ and $Y$ being finite-dimensional vector spaces with generating polyhedral cones, this case was studied in \cite[Lemma~9]{SchVid1970} (see also \cite[Remark~1.7.10(iv)]{KalGaa2019}) and \cite[Section 2.6.1]{KalGaa2019}.

Various examples are discussed at the end of Section~\ref{sec:repr-op-space}.
In Section~\ref{sec:prelim}, we discuss preliminaries on ordered vectors spaces and ordered normed spaces. 
Section~\ref{sec:repr-general} contains a functional representation result 
about ordered normed spaces whose cone has non-empty interior (Theorem~\ref{thm:vcl-general}); 
we will use this result to prove Theorem~\ref{thm:vcl-ops}.

\section{Some facts on ordered normed spaces}
\label{sec:prelim}

We recall some facts from the theory of ordered vector spaces, 
mainly to fix the terminology and notation which is not always consistent throughout the literature.

Let $V$ be a real vector space. 
A \emph{wedge} in $V$ is a non-empty subset $V_+$ 
that satisfies $\alpha V_+ + \beta V_+ \subseteq V_+$ for all real scalars $\alpha, \beta \ge 0$. 
If, in addition, $V_+ \cap (-V_+) = \{0\}$, then the wedge $V_+$ is called a \emph{cone}. 
A real vector space $V$ endowed with a wedge $V_+$ is called a \emph{pre-ordered vector space}.
It is called an \emph{ordered vector space} if $V_+$ is even a cone. 
The wedge $V_+$ in a pre-ordered vector space $V$ induces a reflexive and transitive relation 
(i.e., a pre-order) $\le$ on $V$ given by 
$x \le y$ if and only if $y-x \in V_+$.
This pre-order is an order if and only if the wedge is even a cone. 
The wedge $V_+$ in a pre-ordered vector space $V$ is called \emph{generating} if its linear span $V_+-V_+$ equals $V$. 
This is equivalent to the space $V$ being \emph{directed} which means that for any two elements $v_1,v_2 \in V$ 
there exists an element of $V$ that dominates both $v_1$ and $v_2$.
An ordered vector space $V$ is called a \emph{vector lattice} or a \emph{Riesz space} 
if any two elements of $V$ have a supremum.
Standard references for the theory of Riesz spaces are, for instance, \cite{AliBur2006, MeyerNieberg1991, Schaefer1974}.
A linear map $T\colon V \to W$ between two pre-ordered vector spaces is called \emph{positive} 
if, for all $v \in V$, the inequality $v \ge 0$ implies $T v \ge 0$. In this case,  
we write $T \ge 0$. 
The map $T$ is called \emph{bipositive} if, for all $v \in V$, the inequality $v \ge 0$ is equivalent to $T v \ge 0$.

Let $V$ be a pre-ordered vector space. 
For $v_1, v_2 \in V$, the set $[v_1, v_2] := \{x \in V\colon \, v_1 \le x \le v_2\}$ 
is called the \emph{order interval} between $v_1$ and $v_2$. 
A vector $v \in V_+$ is called an \emph{extremal vector} or an \emph{extremal element} of $V_+$ 
if all elements of $[0,v]$ are scalar multiples of $v$. 

By a (pre-)ordered normed space we mean a normed space $X$ 
that is also a (pre-) ordered vector space and whose wedge $X_+$ is closed. 
Note that the order on such a space is compatible with the topology in the sense that 
if two sequences $(x_n)$ and $(y_n)$ in $X$ converge to points $x$ and $y$ in $X$, 
respectively, and $x_n \le y_n$ for all indices $n$, then also $x \le y$.
Let $X$ and $Y$ be pre-ordered normed spaces and let $\calL(X,Y)$ 
denote the space of bounded linear operators from $X$ to $Y$ endowed with the operator norm. 
Let $\calL(X,Y)_+$ denote its subset of positive operators. 
This set is a closed wedge and thus turns $\calL(X,Y)$ into a pre-ordered normed space.
If $Y_+$ is a cone and $X_+$ is \emph{total}, which means that its span is norm dense in $X$, 
then $\calL(X,Y)_+$ is also a cone.
The norm on an ordered normed space is called \emph{semi-montone} -- or the cone is called \emph{normal} -- 
if there exists a real number $C > 0$ such that for all $x_1,x_2 \in X$ 
the order inequality $0 \le x_1 \le x_2$ implies the norm inequality $\norm{x_1} \le C \norm{x_2}$.
The norm dual space $X'$ of a pre-ordered normed space $X$ 
is itself a pre-ordered normed space with respect to the \emph{dual wedge} 
$X'_+ = \{x' \in X'\colon \; \langle x', x \rangle \ge 0 \text{ for all } x \in X_+\}$, i.e., $X_+'$ consists of those continuous functionals that are positive as maps from $X$ to $\bbR$ 
(where $\bbR$ is endowed with its usual order). 
We call a functional $x' \in X'_+$ an \emph{extremal functional} if it is an extremal element of $X_+'$.

For an ordered vector space $V$,
a vector lattice $W$ together with a bipositive linear map $\Phi\colon V \to W$ is called a \emph{vector lattice cover} 
if each $w \in W$ is the infimum of all those elements of the range $\Phi[V]$ that dominate $w$.
An ordered vector space $V$ is called a \emph{pre-Riesz space} if it has a vector lattice cover. 
The theory of pre-Riesz spaces goes back to \cite{Haa1993}; 
the state of the theory as of 2019 is described in \cite{KalGaa2019}. 
For our purposes we mainly note that every ordered normed vector space with generating cone is a pre-Riesz space; 
this follows, for instance, from \cite[Proposition 2.2.3]{KalGaa2019}.

It is an intriguing feature of pre-Riesz spaces that they admit a reasonable concept of disjointness. 
Let $V$ be a pre-Riesz space. 
Two elements $v_1,v_2 \in V_+$ are called \emph{disjoint} if there infimum in $V$ exists and equals $0$. 
More generally, two elements $v_1, v_2 \in V$ are called \emph{disjoint} 
if $\{v_1+v_2, -v_1-v_2\}$ has the same upper bounds as $\{v_1-v_2,v_2-v_1\}$. 
This is consistent with the definition for the case $v_1, v_2 \in V_+$. 
If $V$ is a vector lattice, then disjointness of $v_1$ and $v_2$ is equivalent 
to the usual vector lattice version of disjointness, i.e., to the property $\lvert v_1 \rvert \land \lvert v_2 \rvert = 0$.
Moreover, if a vector lattice $W$ together with a bipositive map $\Phi\colon V \to W$ is a vector lattice cover of $V$, 
then $v_1,v_2 \in V$ are disjoint if and only if $\Phi(v_1)$ and $\Phi(v_2)$ are disjoint in $W$. 
Thus, having a concrete vector lattice cover of $V$ available often makes it easier to check 
whether two elements of $V$ are disjoint.
This will be one of the main themes throughout the article.

Let $(V,V_+)$ be an Archimedean ordered vector space and let $v$ be an order unit in $V$, i.e., 
for every $x\in V$ there is a real number $\lambda$ such that $x\in [-\lambda v,\lambda v]$. 
Let $V$ be equipped with the order unit norm 
\[
	\norm{x}_v:=\inf\big\{\lambda\in [0,\infty)\colon\ x\in [-\lambda v,\lambda v] \big\}
\] 
for $x\in V$, for details see \cite[Section 1.5.3]{KalGaa2019}. 
In this case, $V$ is called an \emph{order unit space}. 
Every order unit space is an ordered  normed vector space and hence, in particular, a pre-Riesz space. 
The functional representation of such an order unit space $V$ is given by means of the weakly-$\ast$ compact convex set 
\begin{equation}
	\label{equ:Sigma}
	\Sigma
	:=
	\{\varphi\in V'_+ \colon\  \varphi(v)=1\}
\end{equation}
and the set $\extreme \Sigma$ of the  extreme points of $\Sigma$. 
Its weak-$\ast$ closure $\overline{\extreme \Sigma}$ is a compact Hausdorff space 
with respect to the weak-$\ast$ topology. 
Consider the map 
\begin{equation}
	\Phi \colon V \to \mathrm{C}\big(\overline{\extreme \Sigma}\big),
	\quad 
	x \mapsto (\varphi\mapsto\varphi(x)),
\end{equation}
i.e., $\Phi$ is the composition of the canonical embedding from $V$ into $V''$ and the restriction to $\overline{\extreme \Sigma}$.
The map $\Phi$ is a bipositive linear map, and hence injective \cite[Section 2.5]{KalGaa2019}, 
and it is called the \emph{functional representation} of $V$. 
The functional representation of $V$ is always a vector lattice cover, 
as the following proposition shows; 
for its proof see \cite[Theorem 2.5.9]{KalGaa2019}.

\begin{proposition} 
	\label{pro:vlc}
	If $V$ is an order unit space and $\Sigma \subseteq V'_+$ is given as in~\eqref{equ:Sigma}, 
	then the functional representation $\Phi: V \to \mathrm{C}\big(\overline{\extreme \Sigma}\big)$ 
	is a vector lattice cover of $V$.
\end{proposition}

We will generalize this proposition in Theorem~\ref{thm:vcl-general} below. 
To this end we need the following result which can, for instance, 
be found in \cite[Propositions~1.5.11 and~1.5.17]{KalGaa2019}.

\begin{proposition}
	\label{pro:innerpoint}
	Let $V \not= \{0\}$ be an ordered normed space whose cone has an interior point $v$. 
	Then $v$ is an order unit and the order unit norm $\norm{\argument}_v$ 
	is weaker than the original norm $\norm{\argument}$ on $V$. 
	Moreover, every positive functional on $V$ is continuous with respect to both norms.
\end{proposition}

Note that the two norms in Proposition \ref{pro:innerpoint} need not be equivalent, see, e.g., \cite[Example 1.5.12]{KalGaa2019}.
For a detailed discussion in ordered Banach spaces, see \cite[Proposition~2.11]{GlueckWeber2020}.

For an ordered vector space $(V,V_+)$, a set $S'$ of positive linear functionals from $V$ to $\mathbb{R}$ 
is said to \emph{determine positivity} if, for every $v\in V$, 
one has $v\geq 0$ if $\varphi(v)\geq 0$ for all $\varphi\in S'$. 
Note that, if $V$ is an order unit space and $\Sigma \subseteq V'_+$ is given as in~\eqref{equ:Sigma}, 
then $\extreme \Sigma$ determines positivity.

We will make use of the following two simple propositions.

\begin{proposition} 
	\label{prop:zeus}
	Let $X$ be an ordered normed vector space and let $S' \subseteq X'_+$. 
	The following are equivalent: 
	\begin{enumerate}[label=\upshape(\roman*)]
		\item\label{prop:zeus:itm:det-pos} 
		The set $S'$ \emph{determines positivity}, i.e., 
		for every $x \in X$, the inequality $\langle x', x \rangle \ge 0$ for all $x ' \in S'$ 
		implies that $x \ge 0$.
		
		\item\label{prop:zeus:itm:dense} 
		The positive-linear hull of $S'$ is weak-$\ast$ dense in the dual wedge $X'_+$.
	\end{enumerate}
\end{proposition}

\begin{proof}
	\begin{description}
		\item[\normalfont\Implies{prop:zeus:itm:det-pos}{prop:zeus:itm:dense}] 
		We show the contraposition.
		Let $C' \subseteq X'_+$ denote the weak-$\ast$ closure of the positive-linear hull of $S'$ 
		and assume that there exists a point $x'_0 \in X'_+$ that is not in $C'$. 
		By the Hahn--Banach separation theorem in the locally convex space $X'$ with the weak-$\ast$ topology, 
		there exists a weak-$\ast$ continuous functional $x'' \colon X' \to \bbR$ and a number $\gamma \in \bbR$ 
		such that $\langle x'', x' \rangle \ge \gamma > \langle x'', x'_0 \rangle$ for all $x' \in C'$. 
		Due to the weak-$\ast$ continuity of $x''$, there exists an element $x \in X$ such that $x''$ is the point evaluation at $x$, see  
		\cite[§~IV.1.2 on pp.\,124--125]{SchaeferWolff1999}. It follows that $\langle x, x' \rangle \ge \gamma > \langle x, x'_0 \rangle$ 
		for all $x' \in C'$. 
		Since $0\in C'$,  we obtain $\gamma\leq 0$. 
		Therefore,
		$x$ is not in $X_+$.
		As $C'$ is a wedge,
		  for every $x'\in C'$ we have $\langle x, x' \rangle\geq 0$, hence $C'$ does not determine positivity.
		
		\item[\normalfont\Implies{prop:zeus:itm:dense}{prop:zeus:itm:det-pos}]
		If $x \in X$ satisfies $\langle x', x \rangle \ge 0$ for all $x' \in S'$, 
		then it follows from~\ref{prop:zeus:itm:dense} that the same inequality holds for all $x' \in X'_+$. 
		As $X_+$ is closed, this implies that $x \in X_+$. 
		\qedhere
	\end{description}
\end{proof}

The following is a primal version of Proposition~\ref{prop:zeus}. 
As its proof is similar, we omit it. 

\begin{proposition} 
	\label{prop:hera}
	Let $X$ be an ordered normed vector space and let $S \subseteq X_+$. 
	The following are equivalent: 
	\begin{enumerate}[label=\upshape(\roman*)]
		\item\label{prop:hera:itm:det-pos} 
		The set $S$ determines positivity of functionals in the sense that, 
		for each $x' \in X'$, the inequality $\langle x', x \rangle \ge 0$ for all $x \in S$ 
		implies that $x' \ge 0$.
		
		\item\label{prop:hera:itm:dense} 
		The positive-linear hull of $S$ is weakly dense (equivalently, norm dense) in $X_+$.
	\end{enumerate}
\end{proposition}

If the cone $V_+$ in an ordered vector space $X$ has a base $D$, then $x\in D$ is extremal if and only if $x$ is an extreme point of $D$, 
see \cite[Lemma 1.5.19]{KalGaa2019}. 
The following result by Krein and Milman is given, e.g., in 
\cite[§II.10.5 on p.\,68]{SchaeferWolff1999}.

\begin{proposition}
	\label{pro:dirk}
	Let $V$ be a locally convex Hausdorff space. 
	Let $\Sigma\subseteq V$ be compact and convex and let $S'\subseteq \Sigma$. 
	If $\Sigma=\overline{\mathrm{conv}(S')}$ then $\extreme \Sigma \subseteq \overline{S'}$.
\end{proposition}

\section{A general functional representation result}
\label{sec:repr-general}

In the following theorem we give a quite explicit construction of a vector lattice cover 
of an ordered normed space whose cone has an interior point. 
This generalizes the functional representation of order unit spaces given in Proposition \ref{pro:vlc} in two respects. 
First, we allow norms that are not order unit norms and, in fact, need not even be semi-monotone. 
Second, we do not need to know all extremal vectors of the dual cone, 
but only a subset of them that is sufficiently large to determine positivity.

\begin{theorem}
	\label{thm:vcl-general}
	Let $Z \not= \{0\}$ be an ordered normed space whose cone $Z_+$ has an interior point $z_0$
	and let $S' \subseteq Z'$ be a subset with the following properties:
	\begin{enumerate}[label=\upshape(\arabic*)]
		\item\label{thm:vcl-general:itm:base} 
		One has $\langle s', z_0 \rangle = 1$ for all $s' \in S'$.
		
		\item\label{thm:vcl-general:itm:extremal} 
		Every element of $S'$ is an extremal vector of $Z'_+$.
		
		\item\label{thm:vcl-general:itm:determine} 
		The set $S'$ determines positivity.
	\end{enumerate}
	Endow the weak-$\ast$ closure $\overline{S'}$ with the weak-$\ast$ topology and consider the composition
	\begin{align*}
		\Phi: \, Z \to Z'' \to \rmC( \overline{S'} )
	\end{align*}
	of the canonical embedding and the restriction to $\overline{S'}$. 
	Then the space $\rmC(\overline{S'})$ together with the map $\Phi$ is a vector lattice cover of $Z$. 
\end{theorem}

\begin{proof}		
	By Proposition~\ref{pro:innerpoint} the vector $z_0$ is an order unit of $Z$ 
	and the set $Z^\sim_+$ of all positive functionals on $Z$ 
	coincides with $Z'_+$ as well as with the cone in the dual space of $(Z, \norm{\argument}_{z_0})$.
	Note that the weak-$\ast$ topology on $Z^\sim_+$ does not depend on the norm chosen on $Z$. 
	In the remainder of the proof, we endow $Z^\sim_+$ with the weak-$\ast$ topology induced by $Z$. 
	Proposition~\ref{pro:vlc} thus gives us a vector lattice cover 
	$\Phi: Z \to \mathrm{C}(\overline{\extreme \Sigma})$ of $Z$,
	where $\Sigma := \{\varphi\in Z^\sim_+ \colon\  \varphi(z_0)=1\}$. 
	So it suffices to show that $\overline{\extreme \Sigma}= \overline{S'}$.
	
	By assumptions~\ref{thm:vcl-general:itm:base} and~\ref{thm:vcl-general:itm:extremal} 
	we have $S'\subseteq \extreme \Sigma$ and thus, $\overline{S'} \subseteq \overline{\extreme \Sigma}$.
	On the other hand, as $S'$ determines positivity by assumption~\ref{thm:vcl-general:itm:determine}, 
	the positive-linear span of $S'$ is weak-$\ast$ dense in $Z'_+$, according to Proposition~\ref{prop:zeus}. 
	Therefore, the convex hull of $S'$ is weak-$\ast$ dense in $\Sigma$. 
	According to Proposition~\ref{pro:dirk} this implies $\extreme \Sigma \subseteq \overline{S'}$.
	Consequently, $\overline{\extreme \Sigma}= \overline{S'}$.
\end{proof}

For the sake of easier reference, let us state the following special case of Theorem~\ref{thm:vcl-general} explicitly. 
The corollary is close to \cite[Theorem~2.5.9]{KalGaa2019}, 
the only difference being that we allow the norm on $Z_+$ to be more general, 
so that we include, in particular, the case where the cone $Z_+$ is not normal.

\begin{corollary}
	Let $Z$ be an ordered normed space whose cone $Z_+$ has an interior point $z_0$
	and let 
	\begin{align*}
		S' := \{z' \in Z'_+ \colon \, \langle z',z_0 \rangle = 1 \text{ and } z' \text{ is extremal in } Z'_+ \}.
	\end{align*}
	Endow the weak-$\ast$ closure $\overline{S'}$ with the weak-$\ast$ topology and consider the composition
	\begin{align*}
		\Phi: \, Z \to Z'' \to \rmC( \overline{S'} ),
	\end{align*}
	of the canonical embedding and the restriction to $\overline{S'}$. 
	Then the space $\rmC(\overline{S'})$ together with the map $\Phi$ is a vector lattice cover of $Z$. 
\end{corollary}

\begin{proof}
	The set $S'$ obviously satisfies the assumptions~\ref{thm:vcl-general:itm:base} and~\ref{thm:vcl-general:itm:extremal} 
	of Theorem~\ref{thm:vcl-general}. 
	Due to the Krein--Milman theorem and Proposition~\ref{prop:zeus}, 
	it also satisfies assumption~\ref{thm:vcl-general:itm:determine}.
\end{proof}

Let us point out again that, in Theorem~\ref{thm:vcl-general}, 
the set $S'$ does not need to consist of all extremal vectors $s'$ of $Z'_+$ 
that satisfy $\langle s', z_0 \rangle=1$, but only of a subset of them that is sufficiently large to determine positivity. 
Recall that one can check disjointness of two elements of a pre-Riesz space by checking their disjointness in a vector lattice cover.  
Theorem~\ref{thm:vcl-general} provides a vector lattice cover consisting of continuous functions on $\overline{S'}$. 
By observing that two continuous functions are disjoint whenever they are disjoint on a dense subset of $\overline{S'}$ 
-- for instance on $S'$ itself -- one thus gets the following corollary.

\begin{corollary}
	\label{cor:disjoint-determine-pos}
	Let $Z$ be an ordered normed space whose cone $Z_+$ has non-empty interior 
	and let $S' \subseteq Z'_+$ be a set of extremal vectors of $Z'_+$ that determines positivity 
	(for instance, the set of all extremal vectors of $Z'_+)$.  
	Two elements $z_1, z_2 \in Z$ are disjoint if and only if, for every $s' \in S'$, 
	one has $\langle s', z_1 \rangle = 0$ or $\langle s', z_2 \rangle = 0$.
\end{corollary}

\begin{proof}
	Remove $0$ from the set $S'$, if necessary, 
	and replace every element $s' \in S'$ with $\frac{s'}{\langle s', z_0 \rangle}$, 
	where $z_0$ is a fixed interior point of $Z_+$. 
	Then one can apply Theorem~\ref{thm:vcl-general}.
\end{proof}

\section{When does the cone of positive operators have non-empty interior?}

We intend to apply Theorem \ref{thm:vcl-general} to spaces of operators (Section~\ref{sec:repr-op-space}). 
To this end we now characterize when $\calL(X,Y)_+$ contains an interior point. 
For vector spaces $X$ and $Y$, for a vector $y\in Y$ and a linear functional $x'$ on $X$, 
we let $y\otimes x'\colon X \to Y$ denote the linear operator $x\mapsto x'(x)y$, 
which has rank at most $1$.

\begin{theorem}
	\label{thm:interior-op-cone}
	Let $X,Y$ be non-zero ordered normed spaces and assume that the cone $X_+$ is total 
	(hence, the wedge $\calL(X,Y)_+$ of positive bounded linear operators is a cone). 
	The following are equivalent:
	\begin{enumerate}[label=\upshape(\roman*)]
		\item\label{thm:interior-op-cone:itm:ops} 
		The cone $\calL(X,Y)_+$ has non-empty interior in $\calL(X,Y)$.
		
		\item\label{thm:interior-op-cone:itm:spaces} 
		The cone $Y_+$ has non-empty interior in $Y$ and there exists an equivalent norm on $X$ which is additive on $X_+$.
	\end{enumerate}
	If those equivalent assertions are satisfied, 
	then the interior of $\calL(X,Y)_+$ contains a rank-$1$ operator; 
	more specifically, for every interior point $y_0$ of $Y_+$ and every interior point $x_0'$ of $X'_+$, 
	the rank-$1$ operator $y_0\otimes x_0' $ is an interior point of $\calL(X,Y)_+$.
\end{theorem}

\begin{proof}
	\begin{description}
		\item[\normalfont\Implies{thm:interior-op-cone:itm:ops}{thm:interior-op-cone:itm:spaces}] 
		Let $T_0$ be in the interior of $\calL(X,Y)_+$. 
		Then there exists $\varepsilon >0$ such that the $\varepsilon$-ball centered in $T_0$ is contained in $\calL(X,Y)_+$.
	
		We first show that $Y_+$ has non-empty interior.  
		To this end, fix $x\in X_+\setminus\{0\}$ and choose $x'\in X'$ such that $\langle x', x\rangle = 1$.
		Define $y_0:=T_0 x$, which is in $Y_+$. 
		We claim that the $\frac{\varepsilon}{\|x'\|}$-ball in $Y$ centered at $y_0$ is contained in $Y_+$. 
		Indeed, let $y$ be in this ball and define $T:=T_0+(y-y_0)\otimes x'$. 
		Then $Tx=y$ and $\|T-T_0\|\leq \|y-y_0\|\,\|x'\|<\varepsilon$. 
		So $T$ is positive and hence $y\in Y_+$.
		
		Now we show that there exists an equivalent norm on $X$ that is additive on $X_+$. 
		According to a result by Rubinov \cite[Theorems VII.1.1 and VII.3.1]{Wul2017} 
		it suffices to show that $X'$ has non-empty interior.
		Fix $y'\in Y_+'\setminus\{0\}$ and $y\in Y$ such that $\langle y', y\rangle = 1$. 
		Define $x_0'\colon x\mapsto \langle y', T_0x\rangle$ which is in $X_+'$. 
		We claim that the $\frac{\varepsilon}{\|y\|}$-ball in $X'$ centered at $x_0'$ is contained in $X_+'$. 
		Indeed, let $x'$ be in this ball and define $T:=T_0+y\otimes(x'-x_0')$. 
		Then $T\in \calL(X,Y)_+$, and for every $x\in X_+$ we have 
		$0\leq \langle y',Tx\rangle=\langle x',x\rangle$, which implies that $x'\in X'_+$.		
		
		\item[\normalfont\Implies{thm:interior-op-cone:itm:spaces}{thm:interior-op-cone:itm:ops}]
		Without loss of generality, we assume that the norm on $X$ is additive on $X_+$. 
		Then there exists a functional $\one \in X'$ that satisfies $\langle \one, x \rangle = \norm{x}$ for all $x \in X_+$. 
		Indeed, the functional $\one$ on $\operatorname{span}{X_+}$ given by 
		$\langle \one,x-y\rangle := \|x\|-\|y\|$ for all $x,y \in X_+$ is well-defined and continuous 
		with respect to the norm inherited from $X$. 
		Hence, it can be continuously extended to the closure of $\operatorname{span}{X_+}$ and then, 
		by the Hahn--Banach extension theorem, to all of $X$.
		
		Let $y_0$ be in the interior of $Y_+$. By Proposition \ref{pro:innerpoint}, 
		there exists a real number $C > 0$ such that for every $y\in Y$ we have $\|y\|_{y_0}\leq C \|y\|$.
		We show that $y_0 \otimes \one$ is an interior point of $\calL(X,Y)_+$.		
		Indeed, let $T \in \calL(X,Y)$ be such that $\norm{T} \le \frac{1}{C}$. 
		For every $x \in X_+$, one has 
		\begin{align*}
			Tx 
			\le 
			\norm{Tx}_{y_0} y_0 
			\leq C\norm{Tx} y_0\leq
			\norm{x} y_0 
			= 
			\langle \one, x \rangle y_0
			= 
			(y_0 \otimes \one) x
			,
		\end{align*}
		so $T \le y_0 \otimes \one$. 
		As the same argument applies to $-T$, we conclude that $-y_0 \otimes \one \le T \le y_0 \otimes \one$. 
		Consequently, $y_0 \otimes \one$ is an order unit, 
		and the $\frac{1}{C}$-ball of the operator norm is contained in the unit ball of the order unit norm. 
		Thus, $y_0 \otimes \one$ is an interior point of $\calL(X,Y)_+$ with respect to the operator norm.
	\end{description}
	\smallskip 
	
	Finally, assume that the equivalent assertions~\ref{thm:interior-op-cone:itm:ops} 
	and~\ref{thm:interior-op-cone:itm:spaces} are satisfied. 
	If $y_0$ is an interior point of $Y_+$, then $y_0 \otimes \one$ is an interior point of $\calL(X,Y)_+$ 
	as shown in the second implication. 
	If $x_0'$ is an interior point of $X'_+$, then there is a number $\delta > 0$ such that $x_0' \ge \delta \one$.
	Hence, $y_0 \otimes x_0' \ge \delta (y_0 \otimes \one)$, which implies that $y_0 \otimes x_0'$ is an interior point of $\calL(X,Y)_+$.
\end{proof}

A related result in finite dimensions can be found in \cite[Proposition~2.2]{GluHoe2023}.

\section{Extremal functionals on operator spaces}
\label{sec:extremal}

Let $X$ and $Y$ be ordered normed vector spaces 
and let $\calC(X,Y)$ denote the closure of the subspace of finite rank operators in $\calL(X,Y)$. 
We set $\calC(X,Y)_+ := \calC(X,Y) \cap \calL(X,Y)_+$. 
If $Y$ is a Banach space with the approximation property, 
then $\calC(X,Y)$ coincides with the space of compact linear operators from $X$ to $Y$. 
For all $x \in X$ and all $y' \in Y'$, we let $y' \otimes x \in \calC(X,Y)'$ denote the functional 
that is given by 
\begin{align*}
	\langle y' \otimes x, \, T \rangle := \langle y', Tx \rangle
\end{align*}
for all $T \in \calC(X,Y)$. 
We note that the norm of the functional $y' \otimes x$ is given by
\begin{align}
	\label{eq:norm-tensor-functional}
	\norm{y' \otimes x} = \norm{y'} \norm{x};
\end{align}
indeed, the inequality $\le$ is clear and the converse inequality 
can be checked by applying the functional to rank-$1$ operators.

The goal of this section is to show the following theorem. 
A related, but simpler finite-dimensional result can be found in \cite[Proposition~2.3]{GluHoe2023}.

\begin{theorem}
	\label{thm:extremal-ops}
	Let $X$ and $Y$ be ordered normed spaces with the following properties: 
	\begin{enumerate}[label=\upshape(\arabic*)]
		\item 
		The cone $X_+$ is total and normal, 
		and every extremal vector $x$ of $X_+$ is also extremal in the bidual cone $X''$.
		
		\item
		The cone $Y_+$ is total.
	\end{enumerate}
	For non-zero vectors $x \in X_+$ and $y' \in Y'_+$, 
	the functional $y' \otimes x \in \calC(X,Y)'_+$ is extremal in $\calC(X,Y)'_+$ 
	if and only if $x$ is extremal in $X'_+$ and $y$ is extremal in $Y_+$.
\end{theorem}

The assumption that every extremal vector $x$ of $X_+$ be also extremal in $X''_+$ 
might seem rather strong or peculiar at first glance. 
But we shall see in Section~\ref{sec:repr-op-space} that this assumption is satisfied in a variety of situations.

To make the proof of Theorem~\ref{thm:extremal-ops} as transparent as possible, 
we phrase it in the more abstract setting of the following result. 
Despite its higher generality, we call the following result a lemma as it is mainly a tool for us 
in order to obtain Theorem~\ref{thm:extremal-ops} as a special case afterwards. 
We use the following terminology: 
if $X,Y,Z$ are pre-ordered vector spaces, a bilinear map $a: X \times Y \to Z$ is called \emph{positive} 
if $a(x,y) \in Z_+$ for all $x \in X_+$ and $y \in Y_+$.

\begin{lemma}
	\label{lem:bilinear}
	Let $X,Y,Z$ be ordered normed spaces with the following properties:
	\begin{enumerate}[label=\upshape(\arabic*)]
		\item 
		The cone $X_+$ is total and normal, 
		and every extremal vector $x$ of $X_+$ is also extremal in the bidual cone $X''_+$.
		
		\item
		The cone $Y_+$ is total.
	\end{enumerate}
	Consider continuous, positive, bilinear maps 
	\begin{align*}
		a\colon \;  &  X  \times Y' \to Z, \\ 
		b\colon \;  &  X' \times Y  \to Z'
	\end{align*} 
	such that the range $b(X' \times Y) \subseteq Z'$ separates the points of $Z$ and such that the property
	\begin{align}
		\label{equ:abxy}
		\big\langle a(x,y'),b(x',y)\big\rangle_{\langle Z, Z' \rangle}
		=
		\langle x',x\rangle_{\langle X', X\rangle}  \cdot  \langle y',y\rangle_{\langle Y', Y \rangle}
	\end{align}
	holds for all $x\in X$, $x'\in X'$ and all $y\in Y$, $y'\in Y'$. 
	
	Then, for every non-zero $x\in X_+$ and every non-zero $y'\in Y'_+$, 
	the vector $a(x,y') \in Z_+$ is extremal in $Z_+$ if and only if $x$ is extremal in $X_+$ and $y'$ is extremal in $Y'_+$.
\end{lemma}

In a purely algebraic situation -- without any topologies involved --
a result in a similar spirit was shown by van Dobben de Bruyn for projective tensor products of cones in \cite[Theorem~3.22]{deBruyn2022}. 
Compare also Section~1 of the classical paper \cite{NamiokaPhelps1969} by Namioka and Phelps. 
Related results about extreme points of unit balls in operator spaces can be found 
in \cite{LimaOlsen1985, RuessStegall1982}.

\begin{proof}[Proof of Lemma~\ref{lem:bilinear}] 
	We begin with some basic observations.
	First note that, as $X_+$ and $Y_+$ are total, $X'_+$ and $Y'_+$ are cones. 
	Moreover, $X'_+$ is generating since $X_+$ is normal \cite[Theorem~4.5]{KrasnoselskiiLifshitsSobolev2989}.
	
	Next we observe that, for every fixed $y'\in Y'_+\setminus\{0\}$, the map $a(\argument, y')\colon X \to Z$ is bipositive. 
	Indeed, let $x\in X$ be such that $a(x,y')\geq 0$. 
	As $Y_+$ is total, we can choose $y\in Y_+$ such that $\left\langle y,y'\right\rangle=1$.
	For every $x'\in X_+'$, one gets $b(x',y)\in Z'_+$ and thus, by~\eqref{equ:abxy}, 
	$\left\langle x,x'\right\rangle\geq 0$. 
	Hence, $x\geq 0$. 
	Analogously, for every fixed $x\in X_+\setminus\{0\}$, the map $a(x,\argument): Y' \to Z$ is bipositive.
	
	Now we show the claimed equivalence. 
	Fix non-zero vectors $x \in X_+$ and $y' \in Y'_+$. 
	\begin{itemize}
		\item[``$\Rightarrow$'']
		Assume that $a(x,y')$ is extremal in $Z_+$. 
		
		Let $0\leq v\leq x$. 
		As $a$ is entry-wise positive, we get $0\leq a(v,y')\leq a(x,y')$. 
		By the extremality of  $a(x,y')$, 
		there is $\alpha\in [0,\infty)$ such that $a(v,y')=\alpha a(x,y')=a(\alpha x, y')$. 
		As $a(\cdot, y')$ is bipositive and, hence, injective, we obtain $v=\alpha x$. 
		Therefore, $x$ is extremal in $X_+$. 
		Analogously one gets that $y'$ is extremal in $Y'_+$.
		
		\item[``$\Leftarrow$'']
		Let $x$ be extremal in $X_+$ and let $y'$ be extremal in $Y'_+$.
		Let $z\in Z$ be such that $0\leq z\leq a(x,y')$. 
		We have to show that $z$ is a multiple of $a(x,y')$.
		To this end, consider the positive bounded linear map $S\colon X' \to Y'$, $x' \mapsto \langle z, b(x',\argument)\rangle$.
		For all $x' \in X'_+$ and $y\in Y_+$ one has 
		\begin{align}
			\label{equ:bilinear-z}
			\langle Sx', y \rangle 
			=
			\langle z, b(x',y)\rangle
			\leq 
			\langle a(x,y'), b(x',y)\rangle
			=
			\langle x',x\rangle
			\langle y',y\rangle
			,
		\end{align}
		so $0 \leq S x'\leq \langle x,x'\rangle y'$ for all $x' \in X'_+$. 
		As $y'$ is extremal in $Y'_+$ we conclude that $S$ maps $X'_+$ into the one-dimensional subspace $\bbR y'$ of $Y'$. 
		Since $X'_+$ is generating in $X$, it follows that even $SX' \subseteq \bbR y'$. 
		Thus, there exists a functional $x_0'' \in X''_+$ such that
		\begin{align}
			\label{equ:S-formula}
			Sx' = \langle x_0'', x' \rangle y'
		\end{align}
		for all $x' \in X'$. 
		Indeed, for $x'\in X'$ there exists $x''_0(x')\in \mathbb{R}$ such that $Sx'=x''_0(x')y'$. 
		As $y'\in Y'_+\setminus\{0\}$, one obtains that $x''_0$ is linear, positive and continuous 
		by the corresponding properties of $S$. 
		
		Next, we show that $x_0'' \le x$ in $X''$. 
		To this end, note that the formula~\eqref{equ:S-formula} for $S$ together with~\eqref{equ:bilinear-z} gives, 
		for all $y \in Y_+$ and all $x' \in X'_+$,
		\begin{align*}
			\langle x_0'', x' \rangle  \langle y', y \rangle 
			=
			\langle Sx', y \rangle
			\le 
			\langle x',x\rangle
			\langle y',y\rangle
			.
		\end{align*}
		Since $y'$ is non-zero and $Y_+$ is total in $Y$, 
		we can find a vector $y \in Y_+$ such that $\langle y', y \rangle > 0$. 
		Hence, we conclude that $\langle x_0'', x' \rangle \le \langle x',x\rangle$ for all $x' \in X'_+$, 
		therefore $x_0'' \le x$.
		
		We know that $x$ is extremal in $X_+$ and thus, by the assumption of the theorem, also in $X''_+$. 
		Thus, there exists a number $\lambda\in [0,\infty)$ such that $x_0''=\lambda x$. 
		For all $x' \in X'$ this gives 
		$\langle z, b(x',\argument)\rangle = Sx' = \langle x_0'', x' \rangle y' = \lambda \langle x', x \rangle y'$
		and, hence, for all $x' \in X'$ and $y \in Y$, 
		\begin{align*}
			\langle z, b(x',y)\rangle 
			= 
			\lambda \langle x', x \rangle \langle y', y \rangle 
			= 
			\langle \lambda a(x,y'),b(x',y)\rangle
			,
		\end{align*}
		where the second equality follows from~\eqref{equ:abxy} again.
		As the range of $b$ separates the points of $Z$ by assumption, 
		we conclude that $z=\lambda a(x,y')$. 
		\qedhere
	\end{itemize}
\end{proof}

\begin{proof}[Proof of Theorem~\ref{thm:extremal-ops}]
	We are going to apply Lemma~\ref{lem:bilinear}; 
	to this end, choose $Z := \calC(X,Y)'$. 
	We first show that the wedge $Z_+$ in this space is a cone. 
	Indeed, since $X_+$ is normal, the cone $X'_+$ is generating. 
	This together with the fact that $Y_+$ is total implies that every finite rank operator in $\calL(X,Y)$ 
	can be approximated in operator norm by a sequence of differences of positive finite-rank operators. 
	In other words, every finite rank operator can be approximated by a sequence from the span of $\calC(X,Y)_+$. 
	Since $\calC(X,Y)$ is the closure of the finite rank operators, we conclude that the cone $\calC(X,Y)_+$ 
	is total in $\calC(X,Y)$ and, hence, the dual wedge is indeed a cone. 
	
	Now consider the bilinear maps
	\begin{align*}
		a\colon X \times Y' \to Z = \calC(X,Y)', 
		\qquad 
		(x,y') \mapsto y' \otimes x
		, 
		\\
		b\colon X' \times Y \to \calC(X,Y) \subseteq \calC(X,Y)'' = Z',
		\qquad 
		(x',y) \mapsto y \otimes x'
		.
	\end{align*}
	Those are continuous and positive, and a brief computation shows that they satisfy the equality~\eqref{equ:abxy}. 
	The span of the range of $b$ is the set of all finite rank operators in $\calL(X,Y)$ 
	and is thus dense in $\calC(X,Y)$; 
	thus, it separates the points of $Z = \calC(X,Y)'$. 
	So all assumptions of Lemma~\ref{lem:bilinear} are satisfied 
	and we obtain the claimed equivalence.
\end{proof}

\section{A functional representation for the closure of the finite rank operators}
\label{sec:repr-op-space}

The goal of this section is to give, 
under appropriate assumptions on $X$ and $Y$, 
a vector lattice cover of the operator space $\calC(X,Y)$ 
defined at the beginning of Section~\ref{sec:extremal}. 
The following is our main result.

\begin{theorem}
	\label{thm:vcl-ops}
	Let $X,Y$ be non-zero ordered normed spaces with the following properties:
	\begin{enumerate}[label=\upshape(\arabic*)]
		\item\label{thm:vcl-ops:itm:X} 
		The cone $X_+$ is total, there exists an equivalent norm on $X$ that is additive on $X_+$, 
		and every extremal vector of $X_+$ is also extremal in the bidual cone $X''_+$.
		Moreover, the convex hull of the extremal vectors in $X_+$ is dense in $X_+$.
		
		\item\label{thm:vcl-ops:itm:Y} 
		The cone $Y_+$ has non-empty interior.
	\end{enumerate}
	For an interior point $y_0$ of $Y_+$ and an interior point $x_0'$ of $X'_+$,  
	define the subset
	\begin{align*}
		S' 
		:=
		\big\{ 
			y' \otimes x \colon \; 
			x \text{ is extremal in } X_+, \; 
			y' \text{ is extremal in } Y'_+, 
			\text{ and } 
			\langle y', y_0\rangle \langle x_0', x \rangle = 1 
		\big\}
	\end{align*}
	of $\calC(X,Y)'$.
	Endow its weak-$\ast$ closure $\overline{S'}$ with the weak-$\ast$ topology and consider the composition
	\begin{align*}
		\Phi\colon \quad \calC(X,Y) \to \calC(X,Y)'' \to \rmC( \overline{S'} ),
	\end{align*}
	of the canonical embedding and the restriction to $\overline{S'}$. 
	Then the space $\rmC(\overline{S'})$ together with the map $\Phi$ is a vector lattice cover of $\calC(X,Y)$. 
\end{theorem}

\begin{proof}
	We are going to apply Theorem~\ref{thm:vcl-general} to the space $Z := \calC(X,Y)$ 
	and only need to check that all assumptions of that theorem are satisfied. 
	
	The assumptions on $X$ and $Y$ imply, according to Theorem~\ref{thm:interior-op-cone}, 
	that the positive cone in the space $\calL(X,Y)$ has non-empty interior 
	and that the rank-$1$ operator $z_0 := x_0' \otimes y_0$ 
	is an interior point of $\calL(X,Y)_+$. 
	Hence, it is also an interior point of $\calC(X,Y)_+$.
		
	It remains to check the assumptions \ref{thm:vcl-general:itm:base}--\ref{thm:vcl-general:itm:determine} 
	that are imposed on the set $S'$ in Theorem~\ref{thm:vcl-general}:
	\begin{itemize}
		\item[\ref{thm:vcl-general:itm:base}]
		The condition $\langle y', y_0\rangle \langle x_0', x \rangle = 1$ in the definition of $S'$ 
		means that $\langle x \otimes y', z_0 \rangle = 1$.
		
		\item[\ref{thm:vcl-general:itm:extremal}]
		Theorem~\ref{thm:extremal-ops}, which is applicable due to our assumptions on $X$ and $Y$, 
		shows that $S'$ consists of extremal vectors only.
		
		\item[\ref{thm:vcl-general:itm:determine}]
		To check that $S'$ determines positivity, let $T \in \calC(X,Y)$ be an operator 
		such that $\langle y' \otimes x, T \rangle = \langle y', Tx \rangle \ge 0$ 
		for all $y' \otimes x \in S'$. 
		Consider an extremal vector $x \in X_+$ of norm $1$.
		Since $Y_+$ has non-empty interior, by Proposition \ref{pro:innerpoint}, 
		the set of all extremal vectors of $Y'_+$ of norm $1$ determines positivity of elements of $Y$, 
		so it follows that $Tx \ge 0$. 
		Since we assumed the convex hull of the extremal vectors in $X_+$ to be dense in $X_+$, 
		we thus conclude that $T \ge 0$. 
		\qedhere
	\end{itemize}
\end{proof}

While the assumptions on $X$ in Theorem~\ref{thm:vcl-ops} are rather strong, 
there are still a variety of spaces for which they are satisfied; 
see the examples at the end of this section. 
A nice consequence of Theorem~\ref{thm:vcl-ops} is that it gives us a concrete method 
to check whether two elements of $\calC(X,Y)$ are disjoint.

\begin{corollary} 
	\label{cor:disjoint-operators}
	In the setting of Theorem~\ref{thm:vcl-ops}, 
	two operators $T_1, T_2\in \mathcal{C}(X,Y)$ are disjoint if and only if, 
	for all extremal vectors $x\in X_+$ and $y' \in Y'_+$, 
	one has $\langle y', T_1x \rangle = 0$ or $\langle y', T_2x \rangle = 0$.
\end{corollary}

\begin{proof}
	Consider the vector lattice cover of $\calC(X,Y)$ given by the map $\Phi$ from Theorem \ref{thm:vcl-ops}. 
	Then $T_1$ and $T_2$ are disjoint if and only if 
	$\Phi(T_1)$ and $\Phi(T_2)$ are disjoint elements of the vector lattice $\rmC( \overline{S'} )$. 
	Since $S'$ is dense in $ \overline{S'} $, 
	the latter property is in turn equivalent to the function $\Phi(T_1) \Phi(T_2)$ vanishing on $S'$, 
	which proves the claim.
\end{proof}

Note that the result for functionals in  Corollary~\ref{cor:disjoint-determine-pos} 
can be re-obtained as a special case of Corollary~\ref{cor:disjoint-operators} 
by setting $Y := Z$ and $X := \bbR$ in the latter result.

Theorem~\ref{thm:vcl-ops} provides a tractable method to determine the operators in $\mathcal{C}(X,Y)$ that have a modulus in $\mathcal{C}(X,Y)$. With the aid of \cite[Theorem 35]{SteKalGaa2021}, we obtain the following proposition, where the disjointness of the operators in (iii) can be checked by Corollary  \ref{cor:disjoint-operators}.

\begin{proposition} 
	\label{thm:disjoint_operators}
	Consider the setting of Theorem \ref{thm:vcl-ops}. 
	For an operator $T\in \mathcal{C}(X,Y)$, the following statemens are equivalent.
	\begin{enumerate}[label=\upshape(\roman*)]
		\item
		$T$ has a modulus in $\mathcal{C}(X,Y)$.
		
		\item 
		There is $\tau\in\Phi[\mathcal{C}(X,Y)]$ such that
		the pointwise modulus of $\Phi(T)$ restricted to $S'$ equals $\tau|_{S'}$.
		
		\item 
		There are positive disjoint operators $T_1, T_2\in \mathcal{C}(X,Y)$ such that $T=T_1-T_2$.
	\end{enumerate}
\end{proposition}

To describe the set of elements that have a modulus, the notion of a band is useful. 
Recall that, in a pre-Riesz space $Z$, the disjoint complement of a set $B\subseteq Z$ is denoted by $B^\mathrm{d}$. 
A set $B\subseteq Z$ is called a \emph{band} if $B=B^\mathrm{dd}$. 
A reformulation of \cite[Theorem~35]{SteKalGaa2021} gives the following characterization of the set of elements of $Z$ that have a modulus.

\begin{proposition} \label{pro:set_modulus}
	Let $Z$ be a pre-Riesz space.
	 The set of elements in $Z$ that possess a modulus in $Z$ equals 
	\begin{equation}
		\label{set_modulus}
		\bigcup_{B\subseteq Z \text{ band}} \left(B_+- B_+^\mathrm{d}\right).
	\end{equation}
\end{proposition} 

If $Z=\mathcal{C}(X,Y)$ as in Theorem \ref{thm:vcl-ops}, $Z$ is, in particular, 
an order unit space with $\overline{\extreme \Sigma}=\overline{S'}$ as in the proof of Theorem~\ref{thm:vcl-general}. 
By \cite[Theorem~4.4.17]{KalGaa2019}, $B\subseteq Z$ is a band if and only if it is the zero set 
of a so-called \emph{bisaturated} subset of $\overline{\extreme \Sigma}$. 
To calculate the set in \eqref{set_modulus}, it remains to find all bisaturated subsets of $\overline{S'}$. 
This is tractable whenever $S'$ is closed. 
\begin{example} Let $X$ and $Y$ be finite-dimensional vector spaces with generating polyhedral cones  $X_+$ and $Y_+$, respectively. Then $Y'_+$ is generating and polyhedral, and $S'$ is a finite set.
	 Corollary~\ref{cor:disjoint-operators} shows that disjointness of two operators in  $\mathcal{L}(X,Y)$ 
	can be verified by checking finitely many equalities.
	Moreover, there are only finitely many bisaturated subsets of $\overline{S'}$, and they can be calculated directly by means of linear algebra. This gives the set of all bands in $\mathcal{L}(X,Y)$. Proposition \ref{pro:set_modulus}
provides the set of all operators in $\mathcal{L}(X,Y)$ 
 that have a modulus.
\end{example}

We conclude the paper with a discussion of anti-lattices and with a number of examples.
Recall that an Archimedean ordered vector space $Z$ with generating cone is called an \emph{anti-lattice} 
if any two elements of $Z$ have an infimum if and only if they are comparable 
(i.e., one of the elements is larger than the other). 
Equivalently, two non-zero positive elements of the cone are never disjoint \cite[Theorem~4.1.10]{KalGaa2019}.

\begin{corollary}
	\label{cor:anti-lattice-operators}
	Let the assumptions of Theorem~\ref{thm:vcl-ops} be satisfied. 
	The space $\calC(X,Y)$ is an anti-lattice if and only if both $X'$ and $Y$ are anti-lattices. 
\end{corollary}

\begin{proof}
	We start with the following preliminary observation: 
	By applying Corollary~\ref{cor:disjoint-operators} to the special case $Y = \bbR$, 
	we see that two functionals $x_1', x_2' \in X'$ are disjoint if and only if, 
	for all extremal vectors $x$ of $X_+$, one has $\langle x_1', x \rangle = 0$ or $\langle x_2', x \rangle = 0$.
	\begin{itemize}
		\item[``$\Rightarrow$'']
		Assume that $\calC(X,Y)$ is an anti-lattice. 
		We first show that $X'$ is an anti-lattice, so let $x'_1,x'_2 \in X'_+$ be non-zero. 
		Fix a vector $y_0$ in the interior of $Y_+$. 
		Then the positive rank-$1$ operators $y_0 \otimes x'_1$ and $y_0 \otimes x'_2$ are non-zero 
		and thus, they are non-disjoint elements of $\calC(X,Y)$ as we assumed this space to be an anti-lattice. 
		According to Corollary~\ref{cor:disjoint-operators}, there exist extremal vectors 
		$x$ of $X_+$ and $y'$ of $Y'_+$ such that 
		\begin{align*}
			&
			\langle y', y_0 \rangle  \langle x'_1, x \rangle 
			=
			\langle y', (y_0 \otimes x'_1) x \rangle 
			\not= 
			0
			\\ 
			\text{and} \qquad 
			&
			\langle y', y_0 \rangle  \langle x'_2, x \rangle 
			=
			\langle y', (y_0 \otimes x'_2) x \rangle 
			\not= 
			0
			.
		\end{align*}
		In particular, $\langle x'_1, x \rangle \not= 0$ and $\langle x'_2, x \rangle \not= 0$. 
		By the preliminary observation from the beginning of the proof, 
		this implies that $x'_1$ and $x'_2$ are not disjoint.
		
		Now we show that $Y$ is an anti-lattice. 
		To this end, consider non-zero vectors $y_1, y_2 \in Y_+$. 
		Fix a non-zero functional $x'_0 \in X'_+$; 
		such a functional exists since $X$ is non-zero and $X_+$ is a cone. 
		The operators $y_1 \otimes x'_0$ and $y_2 \otimes x'_0$ are non-zero positive elements of $\calC(X,Y)$ 
		and are thus not disjoint in this space, again since $\calC(X,Y)$ is an anti-lattice. 
		Thus, we can again apply Corollary~\ref{cor:disjoint-operators} 
		to find extremal vectors $x$ of $X_+$ and $y'$ of $Y'_+$ such that 
		\begin{align*}
			&
			\langle y', y_1 \rangle  \langle x'_0, x \rangle 
			=
			\langle y', (y_1 \otimes x'_0) x \rangle 
			\not= 
			0
			\\ 
			\text{and} \qquad 
			&
			\langle y', y_2 \rangle  \langle x'_0, x \rangle 
			=
			\langle y', (y_2 \otimes x'_0) x \rangle 
			\not= 
			0
			.
		\end{align*}
		In particular, $\langle y', y_1 \rangle \not= 0$ and $\langle y', y_2 \rangle \not= 0$. 
		Corollary~\ref{cor:disjoint-determine-pos}, 
		applied to the set $S'$ of all extremal vectors of $Y'_+$, 
		yields that $y_1$ and $y_2$ are not disjoint.
		
		\item[``$\Leftarrow$'']
		Assume that $X'$ and $Y$ are anti-lattices and let $T_1, T_2 \in \calC(X,Y)_+$ be non-zero. 
		Then their dual operators are also non-zero. 
		Since $Y_+$ is a cone, the dual cone $Y'_+$ is weak-$\ast$ total in $Y'$, 
		so there exist functionals $y_1', y_2' \in Y'_+$ such that $T_1' y_1' \not= 0$ and $T_2' y_2' \not= 0$. 
		As $X'$ is assumed to be an anti-lattice, 
		the functionals $T_1' y_1'$ and $T_2' y_2'$ are not disjoint. 
		Thus, by the preliminary observation from the beginning of the proof, 
		there exists an  extremal point $x$ of $X_+$ 
		such that $\langle T_1' y'_1, x \rangle \not= 0$ and $\langle T_2' y'_2, x \rangle \not= 0$.
		Hence, both elements $T_1 x$ and $T_2 x$ of $Y_+$ are non-zero. 
		
		Since $Y$ is also assumed to be an anti-lattice, it follows that $T_1 x$ and $T_2 x$ are not disjoint. 
		As $Y_+$ has non-empty interior by assumption~\ref{thm:vcl-ops:itm:Y} of Theorem~\ref{thm:vcl-ops}, 
		it follows from the Krein-Milman theorem that the extremal vectors in $Y'_+$ determine positivity.
		Thus, it follows from the characterization of disjointness in Corollary~\ref{cor:disjoint-determine-pos} 
		that there exists an extremal vector $y'$ of $Y'_+$ 
		such that $\langle y', T_1 x \rangle \not= 0$ and $\langle y', T_2 x \rangle \not= 0$. 
		According to Corollary~\ref{cor:disjoint-operators}, this shows that $T_1$ and $T_2$ are not disjoint. 
		\qedhere
	\end{itemize}
\end{proof}

A stronger property than being an anti-lattice is that two non-zero elements -- no matter whether they are positive -- 
are never disjoint. 
The following is a version of Corollary~\ref{cor:anti-lattice-operators} for such spaces. 
Its proof is almost literally the same as the proof of Corollary~\ref{cor:anti-lattice-operators} 
-- one just has to consider general elements instead of positive elements now.

\begin{corollary}
	\label{cor:no-disjoint-operators}
	Let the assumptions of Theorem~\ref{thm:vcl-ops} be satisfied. 
	There is no pair of non-zero disjoint elements in $\calC(X,Y)$ if and only if 
	the same is true in each of the spaces $X'$ and $Y$.
\end{corollary}

The assumption that no pair of non-zero elements is disjoint is satisfied in
various JBW-algebras, see \cite[Section~4]{vanGaansKalauchRoelands2024}.
An example of an anti-lattice that contains two non-disjoint non-zero elements can be found in \cite[Example~4.1.12]{KalGaa2019}.

\begin{examples}
	\label{exas:various-examples}
	The assumptions in Theorem~\ref{thm:vcl-ops} are satisfied in a variety of situations:
	\begin{enumerate}[label=(\alph*)]
		\item\label{exas:various-examples:itm:finite-dim} 
		If $X,Y$ are finite-dimensional ordered vector spaces with closed and generating cones, 
		then the assumptions of Theorem~\ref{thm:vcl-ops} are satisfied.

		\item\label{exas:various-examples:itm:ell-1} 
		The sequence space $X :=\ell^1$ with its usual norm and the componentwise order 
		satisfies assumption~\ref{thm:vcl-ops:itm:X} of Theorem~\ref{thm:vcl-ops}. 
		
		\item\label{exas:various-examples:itm:trace-class} 
		Let $H$ be a Hilbert space. 
		The space $X$ of self-adjoint trace class operators on $H$, 
		endowed with the so-called \emph{Loewner cone} of operators $A \in X$ that satisfy $(v | Av) \ge 0$ for all $v \in H$, 
		satisfies assumption~\ref{thm:vcl-ops:itm:X} of Theorem~\ref{thm:vcl-ops}. 
		
		The extremal vectors of $X_+$ are precisely the strictly positive multiples of the rank-$1$ projections on $H$ in this case. 
		By the spectral theorem for compact self-adjoint operators, their convex hull is dense in $X_+$. 
		Moreover, every strictly positive multiple of a rank-$1$ projection on $H$ is also extremal in $X''_+$. 
		Indeed, $X''_+$ is the cone of positive semi-definite bounded linear operators on $H$ 
		and, thus, $X$ is an order ideal in $X''$.
		
		\item\label{exas:various-examples:itm:centered-cones} 
		Let $X$ be a real Banach space, let $x_0 \in X$ and $x_0' \in X'$ be such that $\langle x_0', x_0 \rangle = 1$ 
		and endow $X$ with the \emph{centered cone}
		\begin{align*}
			X_+ := \{x + rx_0 \colon \, x \in \ker x', \; r \ge \norm{x}\}.
		\end{align*}
		Then $X_+$ is a closed cone with non-empty interior and there exists an equivalent norm on $X$ 
		that is additive on $X_+$; 
		see, e.g., \cite[Section~4.1]{Glueck2016} for details. 
		
		If the space $X$ is reflexive, then it follows that $X$ satisfies 
		assumption~\ref{thm:vcl-ops:itm:X} of Theorem~\ref{thm:vcl-ops}.
	\end{enumerate}
\end{examples}

\begin{examples}
	\label{exas:no-disjoint-operators}
	We give two concrete examples in which there exists no pair of non-zero disjoint operators in $\calC(X,Y)$.
	\begin{enumerate}[label=(\alph*)]
		\item\label{exas:no-disjoint-operators:itm:matrices-loewner} 
		Let $m,n \ge 0$ be integers and let $\bbC^{m \times m}_{\mathrm{sa}}$ and $\bbC^{n \times n}_{\mathrm{sa}}$ denote 
		the spaces of all self-adjoint complex square matrices in dimensions $m$ and $n$, respectively, 
		that we endow with the cones of positive semi-definite matrices.
		
	 	Endow the space $\calL(\bbC^{m \times m}_{\mathrm{sa}}, \bbC^{n \times n}_{\mathrm{sa}})$ 
	 	of linear operators with the cone of positive operators. 
	 	Two non-zero operators in this space are never disjoint.
		 
		\item\label{exas:no-disjoint-operators:itm:spin-factors}
		Let $H$ be a real Hilbert space of dimension $\ge 2$ and endow the Hilbert space $H \times \bbR$ with the 
		\emph{ice cream cone}
		\begin{align*}
			\{(h,\alpha)\colon \; \alpha \ge \norm{h}\}.
		\end{align*}
		The space $H \times \bbR$ with this cone is sometimes called a \emph{spin factor}. 
		Since $H \times \bbR$ is a Hilbert space, it has the approximation property and, hence, 
		the space $\calC(H \times \bbR, H \times \bbR)$ consists of all compact linear operators on $H \times \bbR$. 
		
		Two non-zero operators in $\calC(H \times \bbR, H \times \bbR)$ are never disjoint.
	\end{enumerate}
\end{examples}

\begin{proof}
	\ref{exas:no-disjoint-operators:itm:matrices-loewner} 
	There are no pairs of non-zero disjoint elements in $\bbC^{m \times m}_{\mathrm{sa}}$ or $\bbC^{n \times n}_{\mathrm{sa}}$, 
	see e.g.\ \cite[Proposition~4.1.11]{KalGaa2019} or \cite[Proposition~4.2]{vanGaansKalauchRoelands2024}.
	Moreover, the space $\bbC^{m \times m}_{\mathrm{sa}}$ is order isomorphic to its dual space endowed with the dual cone. 
	Finally, according to Example~\ref{exas:various-examples}\ref{exas:various-examples:itm:finite-dim} 
	the assumptions of Theorem~\ref{thm:vcl-ops} are satisfied. 
	Hence, Corollary~\ref{cor:no-disjoint-operators} tells us that there is no pair of non-zero disjoint elements 
	in the space $\calL(\bbC^{m \times m}_{\mathrm{sa}}, \bbC^{n \times n}_{\mathrm{sa}}) = 
	\calC(\bbC^{m \times m}_{\mathrm{sa}}, \bbC^{n \times n}_{\mathrm{sa}})$.
	
	\ref{exas:no-disjoint-operators:itm:spin-factors} 
	The space $H \times \bbR$ with the ice-cream cone is order isomorphic to its dual space with the dual cone. 
	Moreover, the ice-cream cone is a special case of a centered cone 
	as defined in Example~\ref{exas:various-examples}\ref{exas:various-examples:itm:centered-cones}, 
	so the assumptions of Theorem~\ref{thm:vcl-ops} are satisfied. 
	Finally, there is no pair of non-zero disjoint elements in $H \times \bbR$ 
	according to \cite[Proposition~4.3]{vanGaansKalauchRoelands2024}. 
	Thus, Corollary~\ref{cor:no-disjoint-operators} shows that the same is true in $\calC(H \times \bbR, H \times \bbR)$. 
\end{proof}

\bibliographystyle{plain}
\bibliography{literature}

\end{document}